\documentclass[a4paper,10pt]{article}
\usepackage{graphics}
\usepackage{latexsym}
\usepackage{times}
\usepackage{graphicx}
\usepackage{amsmath,amssymb,amsthm}

\theoremstyle{plain}
\newtheorem{Thm}{Theorem}[section]
\newtheorem{Lem}[Thm]{Lemma}

\newtheorem*{Example}{Example}
\newtheorem{Cor}[Thm]{Corollary}

\newcommand{\cG}{\ensuremath{\mathcal{G}}}
\newcommand{\cH}{\ensuremath{\mathcal{H}}}

\newcommand{\cK}{\ensuremath{\mathcal{K}}}

\title{On the hypercompetition numbers of hypergraphs}

\author{
\begin{tabular}{c}
{\sc Boram PARK}
\thanks{The author was supported by Seoul Fellowship.} \\
[1ex]
Department of Mathematics Education \\
Seoul National University, Seoul 151-742, Korea \\
{\tt kawa22@snu.ac.kr} \\
\\
{\sc Yoshio SANO}
\thanks{This work was supported by Priority Research Centers Program
through the National Research Foundation of Korea (NRF) funded by
the Ministry of Education, Science and Technology 
(MEST) (No. 2010-0029638).}\\
[1ex]
Pohang Mathematics Institute \\
POSTECH, Pohang 790-784, Korea \\
{\tt ysano@postech.ac.kr}
\end{tabular} }

\date{}

\begin{document}
\maketitle

\begin{abstract}
The competition hypergraph $C{\cH}(D)$ of a digraph $D$
is the hypergraph
such that the vertex set is the same as $D$
and $e \subseteq V(D)$
is a hyperedge if and only if $e$ contains at least $2$ vertices and
$e$ coincides with the in-neighborhood of some vertex $v$
in the digraph $D$.
Any hypergraph 
$\cH$ with sufficiently many isolated vertices is
the competition hypergraph of an acyclic digraph.
The hypercompetition number $hk(\cH)$
of a hypergraph  $\cH$ is defined to be the smallest number
of such isolated vertices.

In this paper, we study the hypercompetition numbers of hypergraphs.
First, we give two lower bounds for the hypercompetition numbers
which hold for any hypergraphs.
And then, by using these results,
we give the exact hypercompetition numbers for
some family of uniform hypergraphs.
In particular, we give the exact value of the hypercompetition number
of a connected graph.
\end{abstract}


\section{Introduction}

All hypergraphs considered in this paper
may have isolated vertices but have no loops,
where
a vertex $v$ in a hypergraph is called {\it isolated} if
$v$ is not contained in any hyperedge in the hypergraph,
and
a hyperedge $e$ in a hypergraph is called a {\it loop} if
$e$ consists of exactly one vertex.
So all the hyperedges of hypergraphs, in this paper,
have at least two vertices.
If $(x,y)$ is an arc of a digraph $D$,
then $x$ is called an {\it in-neighbor} of $y$ in $D$
and $y$ is called an {\it out-neighbor} of $x$ in $D$.
The {\it in-neighborhood} $N_D^-(v)$ of a vertex $v$
in a digraph $D$ is the set of in-neighbors of $v$ in $D$.

The notion of a competition graph was introduced by Cohen~\cite{Cohen1}
in 1968 and has arisen from ecology.
The {\em competition graph} $C(D)$ of a digraph $D$ is the graph which
has the same vertex set as $D$ and has an edge between vertices $u$ and $v$
if and only if there exists a common out-neighbor of $u$ and $v$ in $D$.
Any graph $G$ together
with sufficiently many isolated vertices
is the competition graph of an acyclic digraph.
Roberts \cite{MR0504054} defined the {\em competition number} $k(G)$ of
a graph $G$ to be the minimum number $k$ such that $G$ together with
$k$ isolated vertices is the competition graph of an acyclic digraph.
Since Cohen introduced the notion of a competition graph,
various variations have been defined and studied by many authors
(see the survey articles by  Kim~\cite{Kim93}
and Lundgren~\cite{Lundgren89}).

The notion of a competition hypergraph was introduced by
Sonntag and Teichert \cite{CompHyp}
as a variant of a competition graph.
The {\it competition hypergraph} $C\cH(D)$ of
a digraph $D$ is the hypergraph
such that the vertex set is the same as $D$ and $e \subseteq V(D)$
is a hyperedge if and only if
$e$ contains at least $2$ vertices and
$e$ coincides with the in-neighborhood of some vertex $v$
in the digraph $D$
(see \cite{CompHyp, MR2438037, MR2438038,Son4} for studies on competition
hypergraphs of digraphs).
Any hypergraph 
$\cH$ with sufficiently many isolated vertices is
the competition hypergraph of an acyclic digraph.
The {\it hypercompetition number} $hk(\cH)$
of a hypergraph $\cH$ is defined to be the smallest number
of such isolated vertices.
Though Sonntag and Teichert called
it just the {\it competition number}
of $\cH$ and denoted it by $k(\cH)$,
we use the terminology ``hypercompetition number" and the notation
$hk(\cH)$ to avoid confusion in the case where we regard graphs
as hypergraphs.
A hypergraph $\cH$ is called a {\it graph}
if $|e|=2$ for any hyperedge $e \in E(\cH)$.
The following example shows the difference between the (ordinary)
competition number of a graph
and the hypercompetition number of a graph.

\begin{Example}
{\rm
Let $\cG$ be a triangle, i.e.,
\[
V(\cG)=\{v_1, v_2, v_3\}, \quad
E(\cG)=\{\{v_1, v_2\}, \{v_1, v_3\}, \{v_2, v_3\}\}.
\]
Since the digraph $D$ defined by 
$V(D) = V(\cG) \cup \{ z\}$
and 
$A(D)=\{(v_1,z), (v_2,z),$ $(v_3,z) \}$
is acyclic and
its competition graph $C(D)$ is the graph $\cG$
with one isolated vertex $z$,
the competition number $k(\cG)$ of $\cG$ is at most one.
But the hypercompetition number $hk(\cG)$ of $\cG$ is equal to $2$,
which follows from Theorem \ref{Hyper:thm;2uniform}.
}
\end{Example}

Opsut \cite{MR679638} showed that the computation of
the competition number of an arbitrary graph is an NP-hard problem.
On the other hand, we can show that the hypercompetition number
of a connected graph is computed easily.

In this paper, we study the hypercompetition numbers of hypergraphs.
First we give two lower bounds for the hypercompetition numbers
which hold for any hypergraph, and then we give several formulas
for the hypercompetition numbers
for some families of uniform hypergraphs.

\section{Main results}

We introduce notation and terminologies used in this section.
For a (hyper)graph $\cH$ and a finite set $I$,
we denote by $\cH\cup I$ the (hyper)graph such that
$V(\cH\cup I)=V(\cH)\cup I$ and
$E(\cH\cup I)=E(\cH)$.
The \textit{degree} $\deg_{\cH}(v)$
of a vertex $v$ in a hypergraph $\cH$
is defined to be the number of hyperedges containing the vertex $v$.
We say two vertices $u$ and $v$ are \textit{adjacent} in $\cH$
if there is a hyperedge $e$ in $\cH$ such that $\{u,v\} \subset e$.

A hypergraph $\cH$ is called {\it $r$-uniform}
if each hyperedge of the hypergraph $\cH$ has the same size $r$,
where $2 \leq r \leq |V(\cH)|$. Note that $2$-uniform hypergraphs are graphs.

A sequence $v_0 v_1 \cdots v_k$ of distinct vertices of a hypergraph $\cH$
is called a \textit{path}
if there exist $k$ distinct hyperedges $e_1, e_2, \ldots, e_k$
such that $e_i$ contains $\{v_{i-1}, v_i\} $ for each $1 \le i \le k$.
A sequence $v_0 v_1 \cdots v_k$ of vertices of a hypergraph $\cH$ 
where $v_0 v_1 \cdots v_{k-1}$ is a path and 
$v_0=v_k$ is called a \textit{cycle}. 
We say that $\cH$ is {\it connected}
if there exists a path between any two vertices of $\cH$.
A \textit{connected component} of $\cH$ is a maximal connected
subhypergraph of $\cH$.

For a digraph $D$,
an ordering $v_1, v_2, \ldots, v_n$ of the vertices of $D$
is called an {\it acyclic ordering} of $D$
if $(v_i,v_j) \in A(D)$ implies $i<j$. 
It is well-known that a digraph $D$ is acyclic if and only if 
there exists an acyclic ordering of $D$. 
For simplicity, 
we denote a set $\{ (x,v) \mid x\in S\}$ 
by $S \rightarrow v$.

\subsection{Two lower bounds for the hypercompetition number of a hypergraph}
\label{Hyper:sec:LB}

In this section, we give two lower bounds
for the hypercompetition numbers of hypergraphs.
Opsut \cite{MR679638} showed the following two lower bounds
for competition numbers:
\begin{itemize}
\item[(LB1)]
For any graph $G$,
$k(G) \geq \theta_e(G) - |V(G)| +2$;
\item[(LB2)]
For any graph $G$,
$k(G) \geq \min_{v \in V(G)} \theta (N_G(v))$,
\end{itemize}
where $\theta_e(G)$ is the smallest number of cliques in $G$
that cover the edges of $G$,
$\theta(H)$ is the smallest number of cliques in a graph $H$
that cover the vertices of $H$ and
$N_G(v) := \{ u \in V(G) \mid uv \in E(G) \}$
is the open neighborhood of a vertex v in the graph $G$.

Our first lower bound for hypercompetition numbers,
which corresponds to (LB1), is as follows:

\begin{Thm}\label{Hyper:Thm;general_lower}
Let $\cH$ be a hypergraph. Then
\[
hk(\cH) \geq |E(\cH)|-|V(\cH)|+ \min_{e\in E(\cH)}|e|.
\]
\end{Thm}

\begin{proof}
Let $n$ and $k$ be the number of vertices in a hypergraph $\cH$
and the hypercompetition number $hk(\cH)$
of the hypergraph $\cH$, respectively.
Then there exists an acyclic digraph $D$
such that $C{\cH}(D) = \cH \cup \{z_1, \ldots, z_k\}$. 
Furthermore, $D$ can be chosen such that 
$v_1, v_2, \ldots, v_n,$ $z_1,$ $\ldots,$ $z_k$ 
is an acyclic ordering of $D$. 
Let $l$ be the smallest index 
such that $\{v_1, v_2, \ldots, v_l \}$ contains a hyperedge of $\cH$.
If there is a vertex $v_j$ in the set 
$\{v_1, v_2, \ldots, v_l\}$
such that $v_j$ has at least two in-neighbors in $D$,
then $N_D^{-}(v_j)$ is a hyperedge of $\cH$ 
and so $\{ v_1, \ldots, v_{j-1} \}$ contains a hyperedge of $\cH$, 
which contradicts the choice of $l$.
Therefore,
$|N_D^{-}(v)| \leq 1$ for any $v \in \{v_1, v_2, \ldots, v_l \}$.
Then all neighborhoods of size at least $2$
are in-neighborhoods of vertices in
$V(D) \setminus \{ v_1, v_2, \ldots, v_l \}$.
So $n+k-l \geq |E(\cH)|$, i.e., $k \geq |E(\cH)|-n+l$.
Since $\min_{e\in E(\cH)}|e| \le l$,
we have $|E(\cH)|-n+ \min_{e\in E(\cH)}|e| \leq |E(\cH)|-n+l \leq k$.
\end{proof}

We present our second lower bound for hypercompetition numbers,
which corresponds to (LB2).

\begin{Thm}\label{Hyper:Thm;general_lower2}
Let $\cH$ be a hypergraph. Then
\[
hk(\cH) \geq \min_{v \in V(\cH)} \deg_{\cH}(v).
\]
\end{Thm}

\begin{proof}
Let $n$ and $k$ be the number of vertices in a hypergraph $\cH$
and the hypercompetition number $hk(\cH)$
of the hypergraph $\cH$, respectively.
Let $m:=\min_{v \in V(\cH)} \deg_{\cH}(v)$.
Then there exists an acyclic digraph $D$ such that
$C{\cH}(D) = \cH \cup \{z_1, \ldots, z_k\}$,
and so there is an acyclic ordering
$v_1, v_2, \ldots, v_n, z_1, \ldots, z_k$ of $D$.
Since $v_n$ is contained in at least $m$ hyperedges,
$v_n$ has at least $m$ out-neighbors in $D$.
Thus $m \le k$.
\end{proof}

For a hypergraph $\cH$ with no isolated vertices,
$\min_{v \in V(\cH)} \deg_{\cH}(v) \ge 1$
and so the following corollary holds
(this
is also justified
by the fact that any acyclic digraph has a vertex
which has no out-neighbors).

\begin{Cor} \label{Hyper:cor;isolated_lower}
For a hypergraph $\cH$ with no isolated vertex, $hk(\cH) \ge 1$.
\end{Cor}

\subsection{The hypercompetition numbers of uniform hypergraphs}
\label{Hyper:sec:uniform}

In this subsection, we give the exact values of 
the hypercompetition numbers of several kinds of uniform hypergraphs 
by using results in Subsection \ref{Hyper:sec:LB}. 
Roberts \cite{MR0504054}  showed the following results 
for the (ordinary) competition numbers of graphs: 
\begin{itemize}
\item[(R1)]
For a triangle free connected graph $G$, $k(G) =|E(G)| - |V(G)| +2$;
\item[(R2)]
For a chordal graph $G$, $k(G) \le 1$,
and the equality holds if and only if $G$ has no isolated vertex.
\end{itemize}
The first result (R1) gives a graph family
which satisfies the equality of (LB1)
and the second result (R2) gives a graph family
which satisfies the equality of (LB2).
In this section, we found two hypergraph families
that correspond to (R1) and (R2), respectively.

An ordering $v_1,v_2,\ldots,v_n$ of the vertices
of an $r$-uniform hypergraph $\cH$ is called
an \textit{elimination ordering} of $\cH$
if, for each $r\le i \le n$, the vertex $v_i$ has degree one
in the subhypergraph of $\cH$ induced by $\{v_1,\ldots,v_i\}$.
Note that if an $r$-uniform  hypergraph $\cH$ has 
an elimination ordering then $|E(\cH)|=n-r+1$.

\begin{Lem}\label{Hyper:Lem:elimi:uniform}
Let $n$ and $r$ be positive integers with $r<n$
and $\cH$ be a connected $r$-uniform hypergraph 
with $n$ vertices 
which has an elimination ordering of $\cH$.
Then $hk(\cH) =1$.
\end{Lem}

\begin{proof}
Let $n$ and $t$ be the numbers of vertices and hyperedges
in a hypergraph $\cH$, respectively.
Let $v_1,v_2,\ldots,v_n$ be an {elimination ordering} of $\cH$.
For each $r \leq i \leq n$, let $e_i$ be the unique hyperedge containing $v_i$
in the subhypergraph of $\cH$ induced by $\{v_1,\ldots,v_i\}$. 
We define a digraph $D$ by
\begin{eqnarray*}
V(D) &:=& V(\cH) \cup \{ z \}, \\
A(D) &:=&  \left( \bigcup_{i=r}^{n-1} ( e_i\rightarrow  v_{i+1} )\right)
 \cup ( e_{n}\rightarrow z ) .
\end{eqnarray*}
Then we can check that
$C{\cH}(D) = \cH \cup \{ z \}$
and that $D$ is acyclic.
Thus $hk(\cH) \leq 1$.
By Corollary~\ref{Hyper:cor;isolated_lower}, we have $hk(\cH) \geq 1$.
Hence the theorem holds.
\end{proof}

The following theorem gives a family of hypergraphs
whose hypercompetition numbers satisfy
the equality of Theorem~\ref{Hyper:Thm;general_lower},
that corresponds to (R1).

\begin{Thm}\label{Hyper:Thm:elimi:uniform}
Let $n$ and $r$ be positive integers such that $r<n$,
and $\cH$ be a connected $r$-uniform hypergraph 
with $n$ vertices. 
Suppose that $\cH$ has a spanning subhypergraph $\cH_0$
which has an elimination ordering of $\cH_0$.
Then
\[
hk(\cH) =|E(\cH)|-|V(\cH)|+r.
\]
\end{Thm}

\begin{proof}
Let $n$ and $t$ be the number of vertices and hyperedges
in a hypergraph $\cH$, respectively.
By Lemma~\ref{Hyper:Lem:elimi:uniform},
there exists an acyclic digraph $D_0$
such that $C\cH(D_0) = \cH_0 \cup \{z_1\}$.
If $E(\cH) \setminus E(\cH_0) = \emptyset$, 
then $|E(\cH)|=|E(\cH_0)|=n-r+1$, 
i.e., $hk(\cH)=1$ and the theorem holds. 
Suppose that $E(\cH) \setminus E(\cH_0)\neq \emptyset$.
Let $E(\cH) \setminus E(\cH_0):= \{ e_1, e_2, \ldots, e_{t-n+r-1} \}$.
We define a digraph $D$ by
\begin{eqnarray*}
V(D) &:=& V(\cH) \cup \{ z_1,z_2,\ldots,z_{t-n+r} \}, \\
A(D) &:=&  A(D_0) \cup \left( \bigcup_{i=1}^{t-n+r-1} 
( e_i\rightarrow z_{i+1} ) \right).
\end{eqnarray*}
Then we can check that
$C{\cH}(D) = \cH \cup \{ z_1,z_2,\ldots,z_{t-n+r} \}$
and that $D$ is acyclic.
Thus $hk(\cH) \leq t-n+r$.
By Theorem \ref{Hyper:Thm;general_lower}, we have $hk(\cH) \geq t-n+r$.
Hence the theorem holds.
\end{proof}

Let us consider $2$-uniform hypergraphs, i.e., graphs. 
Opsut \cite{MR679638} showed that the computation of 
the competition number of an arbitrary graph is an NP-hard problem. 
On the other hand, we can show that the hypercompetition number 
of a graph is computed easily from Theorem~\ref{Hyper:Thm:elimi:uniform}. 

\begin{Cor}\label{Hyper:thm;2uniform}
For a connected graph $\cG$,
\[
hk(\cG) = |E(\cG)|-|V(\cG)|+2.
\]
\end{Cor}

\begin{proof}
Let $|V(\cG)|=n$, and $T$ be a spanning tree of $\cG$.
Let $v_n$ be a pendent vertex of $T$, and
then take a pendent vertex $v_{n-1}$ of $T-v_n$.
For each $2\le i \le n-1$,
we take a  pendent vertex $v_i$ of $T-\{v_{i+1},\ldots,v_n \}$.
Then the ordering $v_1, v_2, \ldots, v_n$ of the vertices of $\cG$
is an {elimination ordering} of $\cH$.
By Theorem~\ref{Hyper:Thm:elimi:uniform},
$hk(\cG) = |E(\cG)|-|V(\cG)|+2$.
\end{proof}

A {\it complete $r$-uniform hypergraph} $\cK(n,r)$ is
the hypergraph with $|V(\cH)|=n$ and $E(\cH) = {V(\cH) \choose r}$,
where ${V(\cH) \choose r}$ denotes the family of all $r$-subsets of $V(\cH)$. 
We can obtain the hypercompetition numbers of complete uniform hypergraphs 
as a corollary of Theorem \ref{Hyper:Thm:elimi:uniform}. 

\begin{Cor}
For $2 \leq r \leq n$, it holds that 
\[
hk(\cK(n,r))= {n \choose r} - n +r. 
\]
\end{Cor}

\begin{proof}
Let $2 \leq r \leq n$. 
By Theorem \ref{Hyper:Thm;general_lower}, 
we have $hk(\cK(n,r)) \ge {n \choose r} - n +r$. 
If $r=n$,
since $\cK(n,n)$ is the only $n$-uniform hypergraph 
with $n$ vertices, then 
it trivially holds that 
$hk(\cK(n,n))=1={n \choose n} - n +n$. 
Suppose that $r<n$. 
Let $V(\cK(n,r)) = \{v_1,v_2,\ldots,v_n\}$. 
Then the spanning subgraph $\cH_0$ with hyperedge set 
$\{ \{v_i, v_{i+1}, \ldots, v_{i+r-1} \} \in E(\cK(n,r)) 
\mid 1 \leq i \leq n-r+1\}$ 
has an elimination ordering $v_1, v_2, \ldots, v_n$. 
By Theorem \ref{Hyper:Thm:elimi:uniform}, 
we have $hk(\cK(n,r)) = {n \choose r} - n +r$. 
Hence, the corollary holds. 
\end{proof}

Next, we present a family of hypergraphs whose hypercompetition numbers
satisfy the equality of the inequality
in Theorem~\ref{Hyper:Thm;general_lower2}.
From (R2), it is well known that
since a forest with no isolated vertex is a chordal graph, its
(ordinary) competition number is exactly one. 
We generalize this result to the case for hypergraphs
by showing that for an $r$-uniform hypergraph $\cH$
with no isolated vertices and no cycles, $hk(\cH)=1$.
We need the following lemma.

\begin{Lem}\label{Hyper:Thm;ConnDegree1}
Let $\cH$ be a hypergraph.
If the number of vertices of degree one in $\cH$ is
at least $|E(\cH)|-1$,
then $hk(\cH)\le 1$,
and the equality holds if and only if $\cH$ has no isolated vertex.
\end{Lem}

\begin{proof}
Let $n$ and $t$ be the number of vertices and hyperedges
in a hypergraph $\cH$, respectively.
Let $Q$ be the set of vertices of degree one in $\cH$ and
let $q := |Q|$.
Label the hyperedges of $\cH$ as $\{ e_1, e_2, \ldots, e_t \}$
so that $\{ e_1, \ldots, e_l \}$ is the set of distinct hyperedges
containing a vertex of $Q$.
Since each vertex in $Q$ is contained in
a unique hyperedge of $\cH$, we have $l \le q$.
Label the vertices of $\cH$ as $\{v_1, v_2, \ldots, v_n\}$,
where $\{v_1,\ldots,v_q\}$ are the vertices of $Q$ and
$\{v_1, \ldots, v_l\}$ were chosen from $Q$ such that
$v_i \in e_i$ for $1 \le i \le l$.
Now we define a digraph $D$ by
\[
V(D) := V (\cH) \cup \{v_0\} \quad \text{ and } \quad
A(D) := \bigcup_{i=1}^{t} (  e_i \rightarrow v_{i-1} ).
\]
By definition, $E(C\cH(D)) = \{ e_1,\ldots, e_t\}$,
so $C\cH(D) = \cH \cup \{v_0\}$.
It remains to show that $D$ is acyclic.
We prove it by showing that  $v_n, \ldots, v_1,v_0$
is an acyclic ordering of $D$.
Consider an arc $(x, v_{i-1})$ where $x \in e_i$.
If $1 \le i \le l$ then $e_i \subseteq \{ v_i, v_{l+1}, \ldots, v_n\}$,
and so $x = v_j$ with $j > i - 1$.
On the other hand, if $i > l$, then
$e_i \subseteq \{v_{q+1}, \ldots, v_n\}$,
and again $x = v_j$ with $j > i- 1$. So all the arcs in $D$ are
of the form $(v_j , v_{i-1})$ with $j > i-1$,
and this shows that $v_n, \ldots, v_1,v_0$ is an acyclic
ordering of $D$ and therefore $D$ is acyclic.

For determining when $hk(\cH) = 1$, suppose that $\cH$ has an isolated vertex.
Then the hypergraph $\cH_0$ obtained from $\cH$ by deleting the set $I$
of isolated vertices of $\cH$ also has $q$ vertices of degree one
and $t$ hyperedges.
The above argument shows that $hk(\cH_0) \le 1$.
Thus $hk(\cH) = hk(\cH_0 \cup I) = 0$. 
On the other
hand, if $\cH$ has no isolated vertices, then $hk(\cH) = 1$
by the above argument and Corollary~\ref{Hyper:cor;isolated_lower}.
This proves the lemma.
\end{proof}

The following lemma is well-known.

\begin{Lem}[{\cite[p.392]{hyper}}] \label{Hyper:lem:AcyclicHyper}
Let $\cH$ be a hypergraph and $p$ be the number of
connected components of $\cH$.
Then $\cH$ has no cycle if and only if
\[
\sum_{e \in E(\cH)} (|e|-1)=|V(\cH)| - p.
\]
\end{Lem}

Now  we will show that for an $r$-uniform hypergraph $\cH$
with no isolated vertices and no cycles, $hk(\cH)=1$.

\begin{Thm}\label{Hyper:Thm;3uniform}
Let $r$ be a positive integer with $r\ge 3$,
and $\cH$ be an $r$-uniform hypergraph with no isolated vertex.
If $\cH$ has no cycle, then $hk(\cH)=1$.
\end{Thm}

\begin{proof}
We prove by induction on the number of connected components of $\cH$.
Suppose that $\cH$ is a connected hypergraph.
Let $n$ and $t$ be the numbers of vertices and hyperedges in $\cH$,
respectively.
Since $|e|=r$ for any $e\in E(\cH)$ and $\cH$ has no cycle,
we obtain by Lemma \ref{Hyper:lem:AcyclicHyper} that
\begin{equation}\label{Hyper:eq:3uniform1}
n=(r-1)t+1.
\end{equation}
Also we have
\begin{equation}\label{Hyper:eq:3uniform2}
\sum_{v \in V(\cH)} \deg_{\cH}(v) =\sum_{e\in E(\cH)} |e| = rt.
\end{equation}
Since $\cH$ is connected, $\deg_{\cH}(v)\ge 1$ for any $v \in V(\cH)$.
Let $q$ be the number of vertices of degree one in $\cH$.
Then,
\begin{equation}\label{Hyper:eq:3uniform3}
\sum_{v \in V(\cH)} \deg_{\cH}(v) \ge 2(n-q) + q=2n-q.
\end{equation}
By (\ref{Hyper:eq:3uniform1}), (\ref{Hyper:eq:3uniform2}), 
and (\ref{Hyper:eq:3uniform3}),
we have 
\[
q \geq 2n-rt = 2(rt-t+1)-rt = (r-2)t+2 \geq t+2. 
\]
By Lemma \ref{Hyper:Thm;ConnDegree1},
it holds that $hk(\cH) = 1$.

Suppose that the statement holds for hypergraphs 
with $p$ connected components
where $p\ge 1$.
Now suppose that $\cH$ has $p+1$ connected components.
Take a connected component $\cH_1$ of $\cH$.
Let $\cH_2$ be the union of the connected components of $\cH$ 
other than $\cH_1$.
Then $\cH_2$ has $p$ components
and it has no isolated vertex and no cycle.
By induction hypothesis,
we have $hk(\cH_2) =1$.
Then, there exists an acyclic digraph $D_1$ (resp. $D_2$) such that
$C\cH(D_1)=\cH_1 \cup \{z_1\}$ (resp.
$C\cH(D_2)=\cH_2 \cup \{z_2\}$),
where $z_1$ (resp. $z_2$) is a new isolated vertex. 
Without loss of generality, 
we may assume that $N^+_{D_2}(z_2)=\emptyset$. 
Since $D_1$ is acyclic, there exists a vertex $v$ in $D_1$
which has no in-neighbor in $D_1$. 
We define a digraph $D$ by
\begin{eqnarray*}
V(D) &:=& V(\cH) \cup \{z_1\}, \\
A(D) &:=& A(D_1) \cup \left( A(D_2) \setminus
( N_{D_2}^-(z_2) \rightarrow z_2) \right)
\cup (N_{D_2}^-(z_2)\rightarrow  v ).
\end{eqnarray*}
Then $D$ is acyclic and $C\cH(D)=\cH \cup \{z_1\}$ 
and so $hk(\cH)\le 1$. 
Since $\cH$ has no isolated vertex, we have $hk(\cH) \geq 1$
by Corollary \ref{Hyper:cor;isolated_lower}.
Hence $hk(\cH)=1$.
\end{proof}


\end{document}